\documentclass[11pt,noindent]{article}
\usepackage{amsmath}                        
\usepackage{amsthm}                         
\usepackage{amssymb}                        
\usepackage{amsfonts}                       
\usepackage{latexsym}                       
\usepackage{setspace}
\usepackage{appendix}	
\usepackage[dvips]{graphicx}
\usepackage[colorlinks=true,linkcolor=blue,citecolor=blue]{hyperref}
\usepackage[all]{xy}
\theoremstyle{plain}
\newtheorem{teor}{Theorem}[section]
\newtheorem{defi}{Definition}
\newtheorem{lema}[teor]{Lemma}
\newtheorem{prop}[teor]{Proposition}
\newtheorem{coro}[teor]{Corollary}
\newtheorem{rema}[teor]{Remark}

\newcommand{\R}{\mathbb{R}}

\newcommand{\C}{\mathbb{C}}
\newcommand{\A}{\mathcal{A}}

\newcommand{\la}{\langle}
\newcommand{\ra}{\rangle}

\title{\bf Spacelike surfaces in Minkowski $4-$space with a canonical normal null direction}
\author{Victor H. Patty-Yujra \\ [0.1in] E-mail:\  vpattyy@fcpn.edu.bo \\ [0.05in] Carrera de Matem\'atica -- Instituto de Investigación Matemática \\ Facultad de Ciencias Puras y Naturales \\ Universidad Mayor de San Andr\'es, Bolivia }
\date{}

\begin{document}
\maketitle
\hrule 
\begin{abstract} A canonical normal null direction on a spacelike surface in the  four dimensional Minkowski space $\mathbb{R}^{3,1}$ is a parallel vector field $Z$ on $\mathbb{R}^{3,1}$ such that the normal component of $Z$ on the surface is a lightlike vector field. We describe the geometric properties of a spacelike surface endowed with a canonical normal null direction and we obtain some characterizations of these surfaces. Moreover, using their Gauss map we study other properties of these surface: the associated ellipse of curvature and their asymptotic directions. Finally, we give two different ways to create these surfaces, one of them involves a nonlinear partial differential equation. 
%similar to Born-Infeld equation in electrodynamics.
\end{abstract}



\hrule
\section*{Introduction}

Consider the four dimensional Minkowski space $\R^{3,1}$ defined by $\R^4$ endowed with the tensor metric of signature $(3,1),$ 
\begin{equation*}
\la\cdot,\cdot\ra=dx_1^2+dx_2^2+dx_3^2-dx_4^2.
\end{equation*}
A surface $M\subset \R^{3,1}$ is said to be {\it spacelike} if the metric $\la\cdot,\cdot\ra$ induces a Riemannian metric on $M,$ thus, at each point $p$ of a spacelike surface $M,$ the Minkowski space is split as $\R^{3,1}=T_pM\oplus N_pM,$ where the tangent plane $T_pM$ and the normal plane $N_pM$ at $p$ are respectively equipped with a metric of signature $(2,0)$ and $(1,1)$ (see for example \cite{O}).

\begin{defi} We say that a spacelike surface $M\subset \R^{3,1}$ has a canonical normal null direction with respect to a parallel vector field $Z$ in $\R^{3,1}$ if the normal part $Z^\perp$ of $Z$ is a lightlike normal vector field on $M.$
\end{defi}

The previous definition on the notion of {\it canonical normal null direction} is the mean concept in this paper. It makes sense for spacelike submanifolds, not only for surfaces, in the $n-$dimensional Minkowski space. It is inspired in the concept of timelike surfaces with a canonical null direction with respect to a parallel vector field in Minkowski space defined by the principal author and G. Ruiz in \cite{PR}: a canonical null direction on a timelike surface is given as the tangent part of the parallel vector field. We can also related with the notion of canonical principal direction on a surface defined by F. Dillen and his collaborators in \cite{DFV,DMN}: they defined a canonical principal direction as a principal direction of the surface given as the tangent part of the parallel vector field. Finally, E. Garnica, G. Ruiz and O. Palmas in \cite{GPR} investigated the case of hypersurfaces with a canonical principal direction with respect to a closed conformal vector field.

In this paper, we are interested in the description of the geometric properties of a spacelike surface endowed with a canonical normal null direction in the four dimensional Minkowski space.

This paper is organized as follows. In Section \ref{fundamental}, we study the fundamental equations which determine a canonical normal null direction on a spacelike surface in Minkowski space and we get properties about their geometry in terms of a differentiable function and a differential $1-$form on the surface. An important consequence is given in Proposition \ref{minflat}: if the spacelike surface is minimal and has a canonical normal null direction then it is flat and has flat normal bundle. We also characterize these surfaces, in some particular cases, we get that they are ruled surfaces. 

In Section \ref{gausssec}, we describe the Gauss map of a spacelike surface endowed with a canonical normal null direction using bivectors of the Minkowski space and the Grassmannian of the oriented spacelike 2-planes. We also describe the parametrization of the ellipse of curvature associated, we find the mean curvature directions and the asymptotic directions on the surface; for example, we prove that $Z^{\top}$ is an asymptotic direction on the surface and the existence of another asymptotic direction depends on the sign of the Gauss curvature $K.$

In Section \ref{construction}, we give two different forms to building spacelike surfaces endowed with a canonical normal null direction. The first one consists of translation spacelike surfaces in Minkowski space, that is, surfaces given by the sum of curves. The second one being the graphs of differential functions, in this case, we prove in Proposition \ref{e1} that  such surfaces have a canonical normal null direction respect to the vector field $e_1$ if and only if the coefficients of the first fundamental form satisfy the fully nonlinear partial differential equation on a open set of $\mathbb{R}^2$
\begin{equation*}
(1+f_y^2)g_x^2-2f_xf_yg_xg_y -(1-g_y^2)f_x^2 =0.
\end{equation*} 
Finally, using conformal functions over the Lorentz numbers, we construct some particular solutions of this partial differential equation.

\section{Fundamental equations}\label{fundamental}

We consider a spacelike surface $M$ in $\mathbb{R}^{3,1}$ with a given canonical normal null direction $Z.$ Suppose that $Z$ is a unit spacelike vector field, in this case, using the natural decomposition 
\begin{equation*}
Z=Z^\top+Z^\perp \ \in \ TM\oplus TN\simeq \mathbb{R}^{3,1},
\end{equation*} and since $\langle Z^\perp,Z^\perp\rangle=0,$ we have that $\langle Z^\top,Z^\top\rangle=1.$ Here and below we denote by $\langle\cdot,\cdot\rangle$ the metric on the Minkowski space $\mathbb{R}^{3,1},$ on the tangent bundle $TM$ and on the normal bundle $NM.$

We denote by $B:TM\times TM \to NM$ the second fundamental form of the immersion $M\subset \mathbb{R}^{3,1}$ given by 
\begin{equation*}
B(X,Y)=\overline{\nabla}_X Y - \nabla_X Y,
\end{equation*} where $\overline{\nabla}$ and $\nabla$ are the Levi Civita connections of $\mathbb{R}^{3,1}$ and $M,$ respectively. Moreover, if $\nu\in NM,$ $A_{\nu}:TM\to TM$ stands for the symmetric operator such that 
\begin{equation*}
\langle A_{\nu}(X),Y\rangle=\langle B(X,Y),\nu \rangle,
\end{equation*}  for all $X, Y\in TM.$ Finally, we denote by $\nabla^\perp$ the Levi Civita connection of the normal bundle $NM$ of the surface $M.$ 

\begin{prop} Let $M$ be a spacelike surface in $\R^{3,1}$ with a canonical normal null direction $Z,$ then the following formulas are satisfied 
\begin{equation}\label{ec1}
\nabla_X Z^\top = A_{Z^\perp} (X) \hspace{0.3in}\mbox{and}\hspace{0.3in}
\nabla^\perp_X Z^\perp=-B(X,Z^\top)
\end{equation} for all $X\in TM.$ 
\end{prop}
\begin{proof}
Using the Gauss and Weingarten equations of the immersion, we have that
\begin{align*}
0=\overline{\nabla}_X Z &=\overline{\nabla}_X Z^\top + \overline{\nabla}_X Z^\perp \\
&= \left[ \nabla_X Z^\top + B(X,Z^\top) \right] + \left[ -A_{Z^\perp}(X)+\nabla^\perp_X Z^\perp \right]
\end{align*} for all $X\in TM.$ We obtain the results by taking the tangent and normal parts of this equality.
\end{proof}

%Since $Z$ is a unit spacelike vector field and $\langle Z^\perp,Z^\perp\rangle=0,$ we have that the tangent part $Z^\top$ of $Z$ satisfies $\la Z^\top,Z^\top \ra=1.$

\begin{lema}\label{dp}
The following identities are satisfied
\begin{equation*}
A_{Z^\perp}(Z^\top)=0 \hspace{0.3in}\mbox{and}\hspace{0.3in} \nabla_{Z^\top}Z^\top=0.
\end{equation*} In particular, we have 
$\la B(X,Z^\top),Z^\perp \ra=0,$ for all $X\in TM.$
\end{lema}
\begin{proof}
From \eqref{ec1}, we have 
\begin{equation*}
0=X\langle Z^\top,Z^\top\rangle=2\langle \nabla_X Z^\top,Z^\top\rangle=2\langle A_{Z^\perp}(X),Z^\top \rangle=2\langle X, A_{Z^\perp}(Z^\top) \rangle,
\end{equation*} for all $X\in TM.$ Thus $\nabla_{Z^\top}Z^\top=A_{Z^\perp}(Z^\top)=0.$ As a consequence, we obtain
\begin{equation*}
\la B(X,Z^\top),Z^\perp \ra=\la A_{Z^\perp}(X),Z^\top\ra =\la X,A_{Z^\perp} (Z^\top)\ra=0,
\end{equation*} for all $X\in TM.$
\end{proof}

Let us consider $W$ a unit spacelike vector field tangent to $M$ such that $\langle Z^\top,W \rangle=0$ and  $(Z^\top,W)$ is positively oriented.

We define the differential $1-$form $\alpha:TM \to \R$ given by 
\begin{equation}\label{1form}
\alpha(X)=\la B(X,W),Z^\perp \ra
\end{equation} for all $X\in TM.$ For the particular case $X=W,$ we denote by $a:=\alpha(W).$

\begin{lema}\label{levicivita} The Levi Civita connection of $M$ satisfies the following relations:
\begin{equation*}
\nabla_{Z^\top} Z^{\top}=0, \hspace{0.2in} \nabla_{W} Z^{\top}=aW, \hspace{0.2in} \nabla_{Z^{\top}} W=0 \hspace{0.1in}\mbox{and}\hspace{0.1in} \nabla_{W} W=-aZ^\top.
\end{equation*} In particular, $[Z^\top,W]=-aW.$
\end{lema}
\begin{proof}
The first equality was given in Lemma \ref{dp}. Now, since 
%since $\langle Z^\top,Z^\top\rangle=1,$ we have that 
$0=W\langle Z^\top,Z^\top\rangle=2\langle \nabla_W Z^\top,Z^\top\rangle$ and $\langle \nabla_W Z^\top,W\rangle=\langle A_{Z^\perp}(W),W\rangle=\langle B(W,W),Z^\perp\rangle=a,$ therefore
$$\nabla_{W} Z^{\top}=\langle \nabla_W Z^\top,Z^\top\rangle Z^\top+\langle \nabla_W Z^\top,W\rangle W=aW.$$ In a similar way, 
%since $\langle W,Z^\top\rangle=0,$ 
we have $0=Z^\top\langle W,Z^\top \rangle=\langle \nabla_{Z^\top} W,Z^\top \rangle+\langle W,\nabla_{Z^\top} Z^\top \rangle=\langle \nabla_{Z^\top} W,Z^\top \rangle$ and $0=Z^\top\langle W,W \rangle=2\langle \nabla_{Z^\top} W,W \rangle,$ thus   $$\nabla_{Z^{\top}} W=\langle \nabla_{Z^\top} W,Z^\top \rangle Z^\top+\langle \nabla_{Z^\top} W,W \rangle W=0.$$ On the other hand, $0=W\langle W,W\rangle=2\langle \nabla_W W,W\rangle$ and $0=W\langle Z^\top, W\rangle=\langle \nabla_W Z^\top, W\rangle+\langle  Z^\top,\nabla_W W\rangle,$ thus $\langle  Z^\top,\nabla_W W\rangle=-\langle \nabla_W Z^\top, W\rangle=-\langle aW, W\rangle=-a,$ therefore $$\nabla_W W=\langle \nabla_W W,Z^\top\rangle Z^\top+\langle \nabla_W W, W\rangle W=-aZ^\top.$$ Finally, $[Z^\top,W]=\nabla_{Z^\top} W - \nabla_W Z^\top=-aW.$
\end{proof}

%\begin{rema}
%For all $X\in TM$ we have $$\nabla_{Z^\top} X=0.$$ This means the following: every vector $X\in TM$ moves in parallel on an integral curve of $Z^\top.$
%\end{rema}
%
%\begin{coro}
%On the surface $M$ we have
%\begin{equation*}
%\nabla_{Z^\top} X=0 \hspace{0.3in}\mbox{and}\hspace{0.3in} \la \nabla_W X,X \ra=0.
%\end{equation*} In particular, 
%\begin{equation*}
%\la \nabla_Y X, X\ra=0
%\end{equation*} for all $X,Y \in TM.$
%\end{coro}
%
%\begin{rema} For all $X,Y \in TM,$ we have
%\begin{equation*}
%Y\la X,X\ra=0.
%\end{equation*} In particular, the quadratic function $|X|^2=\la X,X\ra:M \to \R$ has null differential.
%\end{rema}
%
%\begin{prop} For all $X\in TM,$ the function 
%\begin{equation*}
%|X|^2:M \to \R,
%\end{equation*} is locally constant.
%\end{prop}
%
%\begin{coro} In a parametrization $(x,y)$ of the surface $M,$ the metric is given by \begin{equation*}
%\la\cdot,\cdot\ra=\lambda^2(dx^2+dy^2)+2\rho dxdy
%\end{equation*} where $\lambda$ is constant and $\rho$ is a function.
%\end{coro}

%\subsection{The curvatures of the surface}

Now, in the following results we describe the curvature tensors of the surface $M.$

\begin{prop}\label{tensor} The curvature tensor $R$ and the normal curvature tensor $R^\perp$ of the surface $M,$ in the basis $(Z^\top,W),$ are given by 
\begin{equation*}
R(Z^\top,W)Z^\top=\left( -Z^\top(a)-a^2 \right)W \hspace{0.3in}\mbox{and}\hspace{0.3in}
R^\perp(Z^\top,W)Z^\perp=-a B(Z^\top,W).
\end{equation*}
\end{prop}

\begin{proof}
Using the equalities in the Lemma \ref{levicivita}, we get
\begin{align*}
R(Z^\top,W)Z^\top &=\nabla_W \nabla_{Z^\top} Z^\top - \nabla_{Z^\top} \nabla_W Z^\top + \nabla_{[Z^\top, W]} Z^\top \\
&= - \nabla_{Z^\top} (aW) + \nabla_{(-aW)} Z^\top \\
&= -Z^\top(a)W-a\nabla_{Z^\top}W-a\nabla_W Z^\top \\
&= -Z^\top(a)W-a(aW) \\
&= \left(-Z^\top(a)-a^2\right) W.
\end{align*} On the other hand, by \eqref{ec1} we obtain
\begin{align*}
R^\perp(Z^\top,W)Z^\perp &=\nabla^\perp_W \nabla^\perp_{Z^\top} Z^\perp - \nabla^\perp_{Z^\top} \nabla^\perp_W Z^\perp + \nabla^\perp_{[Z^\top, W]} Z^\perp \\
&=\nabla^\perp_W \left(-B(Z^\top,Z^\top)\right) - \nabla^\perp_{Z^\top}\left(-B(W,Z^\top)\right)  + \nabla^\perp_{(-aW)} Z^\perp \\
&=-\nabla^\perp_W \left(B(Z^\top,Z^\top)\right) + \nabla^\perp_{Z^\top}\left(B(W,Z^\top)\right)  +a B(W,Z^\top); 
\end{align*} by Codazzi equation and the equalities in the Lemma \ref{levicivita}, we get
\begin{align*}
-\nabla^\perp_W \left(B(Z^\top,Z^\top)\right) + \nabla^\perp_{Z^\top}\left(B(W,Z^\top)\right) &= -(\tilde{\nabla}_W B)(Z^\top,Z^\top)-B(\nabla_W Z^\top,Z^\top)-B(Z^\top,\nabla_W Z^\top) \\ & \ \ +(\tilde{\nabla}_{Z^\top} B)(W,Z^\top)+B(\nabla_{Z^\top} W, Z^\top)+B(W,\nabla_{Z^\top} Z^\top) \\
&=-2aB(W,Z^\top),
\end{align*} this finishes the proof.
\end{proof}

\begin{coro}\label{gaussc} The Gauss curvature of the surface $M$ is given by 
\begin{equation*}
K=\dfrac{\langle R(Z^\top,W)Z^\top,W \rangle}{|Z^\top|^2 |W|^2-\langle Z^\top,W\rangle}=-Z^\top(a)-a^2.
\end{equation*}
\end{coro}

Another way to compute the Gauss curvature of the surface $M$ in terms of the differential $1-$form defined in \eqref{1form}, is given by the following proposition.

\begin{prop} The Gauss curvature of the surface $M$ is given by
\begin{equation*}
K=-d\alpha(Z^\top,W)
\end{equation*}
\end{prop}
\begin{proof}
Using Lemma \ref{dp}, we easily get
\begin{align*}
d\alpha(Z^\top,W)&=Z^\top(\alpha(W))-W(\alpha(Z^\top))-\alpha([Z^\top,W]) \\
&= Z^\top(a)-W(\langle B(Z^\top,W),Z^\perp\rangle)-\alpha(-aW) \\
&= Z^\top(a)+a^2;
\end{align*} the by Corollary \ref{gaussc}, we obtain the result. 
\end{proof}

According to Lemma \ref{dp} and since $Z^\perp$ is a lightlike vector field on the surface, there exist a differential $1-$form $\beta:TM\to \R$ given by 
\begin{equation}\label{beta}
B(X,Z^\top)=\beta(X)Z^\perp
\end{equation} for all $X\in TM.$ This $1-$form $\beta$ allows as to compute the normal curvature of the surface. We consider the lightlike vector field $W'$ normal to $M$ such that $\langle Z^\perp,W'\rangle=1$ and $$\left(Z^\top,W,\frac{Z^\perp+W'}{\sqrt{2}},\frac{Z^\perp-W'}{\sqrt{2}}\right)$$ is an orthonormal and positively oriented frame of $\R^{3,1}.$

\begin{coro}\label{normal}
The normal curvature of the surface $M$ is given by 
\begin{equation*}
K_N=-a\beta(W).
\end{equation*}
\end{coro}
\begin{proof}
From the Ricci equation in the orthonormal frame $(Z^\top,W)$ on $TM,$
\begin{align*}
K_N &=\left\langle \left(A_{\frac{Z^{\perp}+W'}{\sqrt{2}}}\circ A_{\frac{Z^{\perp}- W'}{\sqrt{2}}} - A_{\frac{Z^{\perp}- W'}{\sqrt{2}}}\circ A_{\frac{Z^{\perp}+W'}{\sqrt{2}}} \right)\left(Z^\top\right), W \right\rangle \\
&=-\left\langle \left(A_{Z^{\perp}}\circ A_{W'} - A_{W'}\circ A_{Z^{\perp}} \right)\left(Z^\top\right), W \right\rangle \\
&= \langle R^\perp(Z^\top,W)Z^\perp, W'\rangle. 
\end{align*} Thus we get the desired result by replacing the second equality given in Proposition \ref{tensor} and the definition of the $1-$form $\beta$ in \eqref{beta}.
\end{proof}

Analogously, the normal curvature of the surface is given by the exterior derivative of the differential $1-$form $\beta.$ 
\begin{lema}\label{rel}
The differential $1-$forms $\alpha$ and $\beta$ are related by the identity
\begin{equation*}
Z^\top(\beta(W))-W(\beta(Z^\top))=-2\alpha(W)\beta(W).
\end{equation*}
\end{lema}
\begin{proof} In terms of the $1-$form $\beta,$ the Codazzi equation reads as
\begin{equation*}
-\nabla^\perp_W \left( \beta(Z^\top)Z^\perp \right) + \nabla^\perp_{Z^\top}\left( \beta(W)Z^\perp \right) =-2a\beta(W)Z^\perp,
\end{equation*} (see the last part of the proof of Proposition \ref{tensor}), thus we easily get 
\begin{equation*}
\left[ Z^\top(\beta(W))-W(\beta(Z^\top)) \right]Z^\perp=-2a\beta(W)Z^\perp
\end{equation*} which implies the result.
\end{proof} 
 
\begin{prop}
We have the following formula
\begin{equation*}
K_N=d\beta(Z^\top,W)
\end{equation*}
\end{prop}
\begin{proof} From Lemma \ref{rel}, we get
\begin{align*}
d\beta(Z^\top,W) &=Z^\top(\beta(W))-W(\beta(Z^\top))-\beta([Z^\top,W]) \\
&=-2\alpha(W)\beta(W)-\beta(-aW) \\
&=-2a\beta(W)+a\beta(W) \\
&=-a\beta(W),
\end{align*} and combining with the Corollary \ref{normal}, we obtain the result.
\end{proof}

An alternative way to obtain the Gauss curvature in terms of the $1-$form $\beta$ is given in the following colollary.

\begin{coro}\label{gauss} The Gauss curvature of the surface $M$ is given by
\begin{equation*}
K=a\beta(Z^\top)
\end{equation*}
\end{coro}
\begin{proof} Since $a=\langle B(W,W),Z^\perp \rangle,$ using the expression given in Corollary \ref{gaussc} for the Gauss curvature, we have
\begin{align*}
K=-Z^\top(a)-a^2 &= -Z^\top \langle B(W,W),Z^\perp \rangle-a^2 \\
&= -\langle \nabla_W ^\perp (B(Z^\top,W)), Z^\perp \rangle+ a^2 + a\beta(Z^\top)-a^2.
\end{align*} But, since $W\langle B(Z^\top,W),Z^\perp\rangle=0,$ then  $\langle \nabla_W ^\perp (B(Z^\top,W)), Z^\perp \rangle=0,$ which brings us to the desired result.
\end{proof}

The geometric interpretation of the $1-$form $\beta$ is given in the following proposition.

\begin{prop}
$Z^\perp$ is a parallel normal vector field on $M$ if and only if $\beta\equiv 0.$ In particular, if $Z^\perp$ is a parallel normal vector field on $M,$ then $K=K_N=0,$ {\it i.e.}, $M$ is flat and has a flat normal bundle.
\end{prop}
\begin{proof}
Using the second equality in \eqref{ec1}, and the definition of $\beta,$ we obtain
\begin{equation*}
\nabla^\perp_X Z^\perp=-B(X,Z^\top)=-\beta(X)Z^\perp,
\end{equation*}for all $X\in TM,$ which implies the result. The particular case is a consequence of Corollaries \ref{normal} - \ref{gauss}.
\end{proof}

A partial reciprocal assertion of the particular case in the previous proposition is given as a consequence of the following proposition.

\begin{prop} If $a=\alpha(W)\neq 0,$ then the $1-$form $\beta$ is given by
\begin{equation*}
\beta(X)=\left\la X,\frac{K}{a}Z^\top+\frac{-K_N}{a}W\right\ra
\end{equation*} for all $X\in TM.$ In particular, if $a\neq 0$ and $K=K_N=0$ then $Z^\perp$ is a parallel normal vector field.
\end{prop}
\begin{proof}
Since $a\neq 0,$ the first assertion is a direct consequence of Corollaries \ref{normal}-\ref{gauss}. If $a\neq 0$ and $K=K_N=0,$ then $\beta\equiv 0,$ the previous proposition implies the result.
\end{proof}

\begin{rema}
If we consider the orthonormal frame $(Z^\top,W)$ on $TM,$ the mean curvature vector of the immersion $M\subset \R^{3,1}$ is given by
\begin{equation}\label{mcv}
\vec{H}=\frac{1}{2} \mbox{tr}_{\la\cdot,\cdot\ra} B=\frac{1}{2} \left[ B(Z^\top,Z^\top)+B(W,W) \right]
\end{equation}
\end{rema}

\begin{lema}\label{minima} We have $a=0$ if and only if $\la \vec{H},Z^\perp\ra=0. $
\end{lema}
\begin{proof} It is not difficult to see that
\begin{align*}
\langle\vec{H},Z^\perp\rangle &=\frac{1}{2}\langle B(Z^\top,Z^\top),Z^\perp \rangle+\frac{1}{2}\langle B(W,W),Z^\perp\rangle \\ 
&= \frac{1}{2}\langle \beta(Z^\top)Z^\perp,Z^\perp \rangle+\frac{1}{2}a
\end{align*} which implies the result since $\langle Z^\perp,Z^\perp \rangle=0.$
\end{proof}

The principal relation between the Gauss curvature $K,$ the normal curvature  $K_N$ and the mean curvature vector $\vec{H}$ of $M$ is given in the following proposition.

\begin{prop}\label{minflat}
If the surface $M$ is minimal ({\it i.e.} $\vec{H}=0$), then $M$ is flat and has a flat normal bundle ({\it i.e.} $K=K_N=0$).
\end{prop}
\begin{proof} From Lemma \ref{minima} we have $a=0.$ By Corollaries \ref{gaussc} - \ref{normal} we obtain respectively $K=0$ and $K_N=0.$ 
\end{proof}

The following proposition gives a nice relation between the Gauss curvature and the mean curvature vector.

\begin{prop}\label{relation} The mean curvature vector and the Gauss curvature of the surface $M$ satisfy the following identity
\begin{equation*}
4|\vec{H}|^2-2K=|B(W,W)|^2
\end{equation*}
\end{prop}
\begin{proof}
By a direct computation 
\begin{align*}
4|\vec{H}|^2&=4\langle \vec{H},\vec{H} \rangle \\ &= |B(Z^\top,Z^\top)|^2+2\langle B(Z^\top,Z^\top),B(W,W) \rangle + |B(W,W)|^2 \\
&= |\beta(Z^\top)Z^\perp|^2 +2\langle \beta(Z^\top)Z^\perp,B(W,W) \rangle + |B(W,W)|^2 \\
&= |\beta(Z^\top)|^2|Z^\perp|^2 +2\beta(Z^\top)\langle Z^\perp,B(W,W) \rangle + |B(W,W)|^2 \\
&= 0 + 2\beta(Z^\top)a+|B(W,W)|^2.
\end{align*} 
Since $Z^\perp$ is a lightlike vector and the definition of $\alpha$ (see \eqref{1form}); by Corollary \ref{gauss} we obtain the result.
\end{proof}

To finish this section, we prove a result that let us permits study some characterizations of a spacelike surface endowed with a canonical normal null direction.

\begin{prop}\label{carac} We have $K=K_N=0$ if and only if $a=0$ or $\beta= 0.$
\end{prop}
\begin{proof} If $a=0$ or $\beta= 0,$ by Corollaries \ref{normal} - \ref{gauss} we follow $K=K_N=0.$ Conversely, if $K=0,$ by Corollary \ref{gauss}, $a=0$ or $\beta(Z^\top)=0.$ On the other hand, if $K_N=0,$ by Corollary \ref{normal}, we have $a=0$ or $\beta(W)=0,$ thus  $a=0$ or $\beta= 0.$ 
\end{proof}

As a consequence of the Proposition \ref{carac}, the following special characterization of a spacelike surface with a canonical normal null direction is given.

\begin{teor}\label{teo1} Suppose that $M$ is a spacelike surface in $\R^{3,1}$ with a canonical normal null direction $Z$ such that $a=0$ and $\beta=0.$ Then the surface $M$ can be parametrized by 
\begin{equation}\label{paramet}
\psi(x,y)=\alpha(y)+xZ_0^\top,
\end{equation} where $\alpha$ is a curve in $\mathbb{R}^{3,1}$ with $\alpha'(y)$ orthogonal to the constant vector field $Z_0^\top.$
\end{teor}
\begin{proof}
Since $a=0,$ from Lemma \ref{levicivita} we have $[Z^\top,W]=0,$ thus there is a parametrization $(x,y)\to \psi(x,y)$ of $M$ such that
\begin{equation*}
\dfrac{\partial \psi}{\partial x}(x,y)=Z^\top(\psi(x,y)) \hspace{0.3in}\mbox{and}\hspace{0.3in}
\dfrac{\partial \psi}{\partial y}(x,y)=W(\psi(x,y)).
\end{equation*}
Since $\beta=0,$ we have $B(X,Z^\top)=\beta(X)Z^\perp=0,$ for all $X\in TM;$ on the other hand, since $\nabla Z^\top=0,$ we follow
\begin{equation*}
dZ^\top(X)=\overline{\nabla}_X Z^\top=\nabla_XZ^\top+B(X,Z^\top)=0
\end{equation*} for all $X\in TM,$ thus 
\begin{align*}
Z^\top(\psi(x,y)) &=Z^\top(\psi(0,y))+\int_0^x\dfrac{\partial}{\partial u}Z^\top(\psi(u,y)) du \\
&=Z^\top(\psi(0,y))+\int_0^x dZ^\top \left( \dfrac{\partial\psi}{\partial u}(u,y)\right) du \\
&=Z^\top(\psi(0,y)).
\end{align*} Similarly, $Z^\top(\psi(x,y))=Z^\top(\psi(x,0)),$ therefore $Z^\top(\psi(x,y))=Z_0^\top$ is constant, which in turn implies
\begin{align*}
\psi(x,y)&=\psi(0,y)+\int_0^x\dfrac{\partial \psi}{\partial u}(u,y) du \\
&= \psi(0,y)+\int_0^xZ^\top(\psi(u,y)) du \\
&= \psi(0,y)+\int_0^xZ_0^\top du \\
&= \psi(0,y)+x Z_0^\top.
\end{align*} Writing $\alpha(y)=\psi(0,y),$ we get the characterization \eqref{paramet} of $\psi.$
\end{proof}

With the same ideas we can prove an other similar characterization.

\begin{teor}\label{teo2} Suppose that $M$ is a spacelike surface in $\R^{3,1}$ with a canonical normal null direction $Z$ such that $a=0$ and $\beta(Z^\top)=0.$ Then the surface $M$ can be parametrized by 
\begin{equation*}\label{paramet1}
\psi(x,y)=\alpha(y)+xZ^\top(y),
\end{equation*} where $\alpha$ is a curve in $\mathbb{R}^{3,1}$ and $Z^\top(y)$ denotes the restriction of $Z^\top$ on $\alpha.$  
\end{teor}

\section{The Gauss map of a surface with a canonical normal null direction}\label{gausssec}

\subsection{The Grassmannian of the spacelike planes}

Consider $\Lambda^2\R^{3,1},$ the vector space of bivectors of the Minkowski $4-$space $\R^{3,1}$ endowed with its natural tensor metric of signature $(3,3).$ 

The Grassmannian of the oriented spacelike 2-planes (which passes through the origin) in $\R^{3,1}$ is identified with the submanifold of unit and simple bivectors 
\begin{equation}\label{grass}
\mathcal{Q}=\left\lbrace \eta\in \Lambda^2\R^{3,1} \mid \la\eta,\eta\ra=1,\ \eta\wedge\eta=0 \right\rbrace ,
\end{equation} and the oriented Gauss map of a spacelike surface in $\R^{3,1}$ with the map $G:M\to \mathcal{Q}$ such that \begin{equation}\label{gaussmap}
G(p)=u_1\wedge u_2,
\end{equation} where $(u_1,u_2)$ is an oriented orthonormal basis for the tangent space $T_pM.$ 

The Hodge star operator $\star:\Lambda^2\R^{3,1}\to \Lambda^2\R^{3,1}$ is defined by the relation\begin{equation}\label{hodge}
\la\star\eta,\eta'\ra=\eta\wedge\eta'
\end{equation} for all $\eta, \eta'\in \Lambda^2\R^{3,1},$ where we identify $\Lambda^4\R^{3,1}\simeq\R$ using the canonical volume element $e_1\wedge e_2\wedge e_3\wedge e_4\simeq 1.$ This operator satisfies $\star^2=-Id_{\Lambda^2\R^{3,1}},$ and thus $i:=-\star$ defines a complex structure on $\Lambda^2\R^{3,1}.$

We also define the map $H:\Lambda^2\R^{3,1}\times \Lambda^2\R^{3,1}\longrightarrow \mathbb{C}$ by 
\begin{equation}\label{Hnorm}
H(\eta,\eta')=\la\eta,\eta'\ra+i\ \eta\wedge \eta'
\end{equation} for all $\eta, \eta'\in \Lambda^2\R^{3,1}.$ This map is a $\mathbb{C}-$bilinear map on $\Lambda^2\R^{3,1},$ and the Grassmannian \eqref{grass} remains as
\begin{equation}\label{grass2}
\mathcal{Q}=\left\lbrace \eta\in \Lambda^2\R^{3,1} \mid H(\eta,\eta)=1 \right\rbrace.
\end{equation}

The bivectors 
\begin{equation}\label{base}
\left\lbrace E_1:=e_1\wedge e_2,\hspace{0.2in} E_2:=e_2\wedge e_3,\hspace{0.2in} E_3:=e_3\wedge e_1 \right\rbrace
\end{equation} form an orthonormal basis, with respect to the form $H$ of $\Lambda^2\R^{3,1}$ as a complex $3-$space with signature $(+,+,+).$ Using this basis of $\Lambda^2\R^{3,1},$ the Grassmannian \eqref{grass2} is identified with a complex sphere
\begin{equation}\label{grass3}
\mathcal{Q}=\left\lbrace (z_1,z_2,z_3)\in\mathbb{C}^3 \mid z_1^2+z_2^2+z_3^2=1 \right\rbrace
\end{equation}

\subsection{Spacelike surfaces with a canonical normal null direction} We consider a spacelike surface $M$ in $\R^{3,1}$ endowed with a canonical normal null direction $Z,$ with $|Z^\top|^2=1$ and such that $a=\la B(W,W),Z^\perp\ra\neq 0.$ We recall that $W'$ is a lightlike  vector field normal to $M$ such that $\la Z^\perp, W' \ra=1.$ If we write $B(W,W):=bZ^\perp+aW',$ the vectors \begin{equation}\label{adaptada}
e_1=Z^\top,\ \ e_2=W,\ \ e_3=\dfrac{(a-b)Z^\perp+B(W,W)}{\sqrt{2}a},\hspace{0.1in}\mbox{and}\hspace{0.1in} \ e_4=\dfrac{(a+b)Z^\perp-B(W,W)}{\sqrt{2}a} 
\end{equation} form an oriented and orthonormal basis of $\R^{3,1}$ adapted to the immersion $M\subset \R^{3,1}$; therefore, we can define the orthonormal basis \eqref{base} of $\Lambda^2\R^{3,1}.$

\begin{lema}\label{diferential} The Gauss map of $M$ is given by $G=Z^{\top}\wedge W$ and its differential satisfies 
\begin{equation*}
dG(Z^\top)=\beta(Z^\top)Z^\perp\wedge W+ \beta(W)Z^\top\wedge Z^\perp  \hspace{0.2in}\mbox{and}\hspace{0.2in}
dG(W)=\beta(W)Z^\perp\wedge W+Z^\top\wedge B(W,W).
\end{equation*}
\end{lema}
\begin{proof}
Clearly $G=Z^\top\wedge W.$ The differential of $G$ is given by\begin{equation*}
dG(u)=(\nabla_u Z^\top+B(Z^\top,u))\wedge W + Z^\top \wedge(\nabla_u W+B(W,u))
\end{equation*} for all $u\in T_pM;$ using the identities of Lema \ref{levicivita} and the definition of the $1-$form $\beta$ (see Remark \ref{beta}) we conclude the result.
\end{proof}

We describe now the differential of the Gauss map in terms of the orthonormal basis defined in \eqref{base} of $\Lambda^2\R^{3,1}.$

\begin{prop}\label{diferential2} The differential of the Gauss map $G$ satisfies\begin{equation*}
dG(Z^\top)=-\frac{K+iK_N}{\sqrt{2}a}E_2+\frac{K_N-iK}{\sqrt{2}a}E_3 
\end{equation*} and
\begin{equation*}
dG(W)=\frac{K_N-ia(a-b)}{\sqrt{2}a}E_2- \frac{a(a+b)-iK_N}{\sqrt{2}a}E_3 
\end{equation*}
\end{prop}
\begin{proof}
From Lemma \ref{diferential}, the definitions of the frame adapted to the immersion \eqref{adaptada}, and Corollaries \ref{normal}, \ref{gauss}, we easily get the result.     
\end{proof}

\paragraph{The pull-back of the form $H$ by the Gauss map.} The pull-back of the form $H,$ defined in \eqref{Hnorm}, by the Gauss map $G:M\to \mathcal{Q}\subset \Lambda^2\R^{3,1}$ lets us to define, for all $p\in M,$ the complex quadratic form $G^*H_p:T_pM\to \C$ given by
\begin{equation*}
G^*H_p(u):=H(dG(u),dG(u)).
\end{equation*} This form is analogous to the third fundamental form in the classical theory of surfaces in Euclidean $3-$space. We will describe some properties of this quadratic form for a spacelike surface with a canonical normal null direction.

\begin{lema}\label{disc} If $a\neq 0,$ the complex quadratic form $G^*H$ satisfies the following identities:
\begin{enumerate}
\item\quad $H(dG(Z^\top),dG(Z^\top))=0,$
\item\quad $H(dG(W),dG(W))=2(2|\vec{H}|^2-K)-i(2K_N),$ and
\item\quad $H(dG(Z^\top),dG(W))=-K_N+iK.$
\end{enumerate}
\end{lema}
\begin{proof}The proof of these identities are obtained by a direct computation using the formulas of $dG(Z^\top)$ and $dG(W)$ given in Proposition \ref{diferential2}.
\end{proof}

\begin{prop} If $a\neq 0,$ the discriminant of the complex quadratic form $G^*H$ satisfies
$$\mbox{\sf disc}\ G^*H:=-\det G^*H=-(K+iK_N)^2.$$
\end{prop}
\begin{proof} Using the identities of Lemma \ref{disc}, by a direct computation we get
\begin{equation*}
\det G^*H=-(K_N+iK)^2=(K+iK_N)^2
\end{equation*} which implies the result.
\end{proof}

An other direct consequence of Lemma \ref{disc} is the following result.
\begin{prop} If $a\neq 0,$ the complex quadratic form $G^*H$ in null at every point of $M$ if and only if $K=K_N=|\vec{H}|^2=0$ on $M,$ {i.e.} $M$ is flat, has flat normal bundle, and its mean curvature vector is a lightlike vector.  
\end{prop}

The interpretation of the condition $G^*H\equiv 0$ is the following: for all $p\in M,$ the space $dG_p(T_pM)$ belongs to 
\begin{equation*}
G(p)+\left\lbrace \xi\in \Lambda^2\R^{3,1} \mid H(G(p),\xi)=0=H(\xi,\xi)  \right\rbrace\ \subset\ T_{G(p)}\mathcal{Q};
\end{equation*} this set is the union of two complex lines through $G(p)$ in the Grassmannian $\mathcal{Q}$ of the oriented spacelike planes of $\mathbb{R}^{3,1};$ explicitly, this complex lines are given by \begin{equation*}
G(p)+\mathbb{C}E_2 \hspace{0.3in}\mbox{and}\hspace{0.3in} G(p)+\mathbb{C}E_3.
\end{equation*} In particular, the first normal space in $p\in M$ is $1-$dimensional, {i.e.}  the osculator space of the surface is degenerate at every point $p\in M.$

\subsection{The curvature ellipse.} The curvature ellipse associated to the second fundamental form $B$ of a spacelike surface $M$ in $\R^{3,1}$ is defined as the subset on $N_pM$ 
\begin{equation}
E_p:=\left\lbrace B(u,u) \mid u\in T_pM,\ |u|=1 \right\rbrace
\end{equation}
Suppose that the surface $M$ has a canonical normal null direction $Z.$ Recall that $a=\la B(W,W),Z^{\perp}\ra.$ Thus, in order to describe the ellipse of curvature of $M$ we have two cases to consider: $a\neq 0$ and $a=0.$  

\begin{prop}If $a\neq 0,$ the curvature ellipse is not degenerated, at the basis $(Z^{\perp},W')$ for the normal bundle, and origin in $\vec{H},$ the curvature ellipse is parametrized by the equations
\begin{equation*}
 x =\cos(2\theta) \left[ \frac{K-|\vec{H}|^2}{a} \right]+\sin(2\theta) \left[ -\frac{K_N}{a} \right] \hspace{0.2in}\mbox{and}\hspace{0.2in}
 y=\cos(2\theta)\left[ -\frac{a}{2} \right] 
\end{equation*}
\end{prop}

\begin{proof} We write each tangent vector $u\in T_pM$ as $u=\cos\theta Z^{\top}+\sin\theta W,$ with $0\leq \theta\leq 2\pi.$ By a direct computation we have
\begin{equation}\label{ellipse}
B(u,u)=\vec{H}+\cos(2\theta)\left[B(Z^\top,Z^\top)-\vec{H}\right]+ \sin(2\theta) B(Z^\top,W),
\end{equation} writing these normal vectors at the basis $(Z^{\perp},W')$ and using the definition of the $1-$form $\beta$ (see Remark \ref{beta}), the identities of Corollaries \ref{normal}, \ref{gauss}, and  Proposition \ref{relation}, we easily get the result.
\end{proof}

The curvature ellipse in a point $p\in M$ such that $a>0,$ $K-|\vec{H}|^2>0,$ and $K_N<0$ is given in the following figure:

\begin{center}
\includegraphics[scale=0.5]{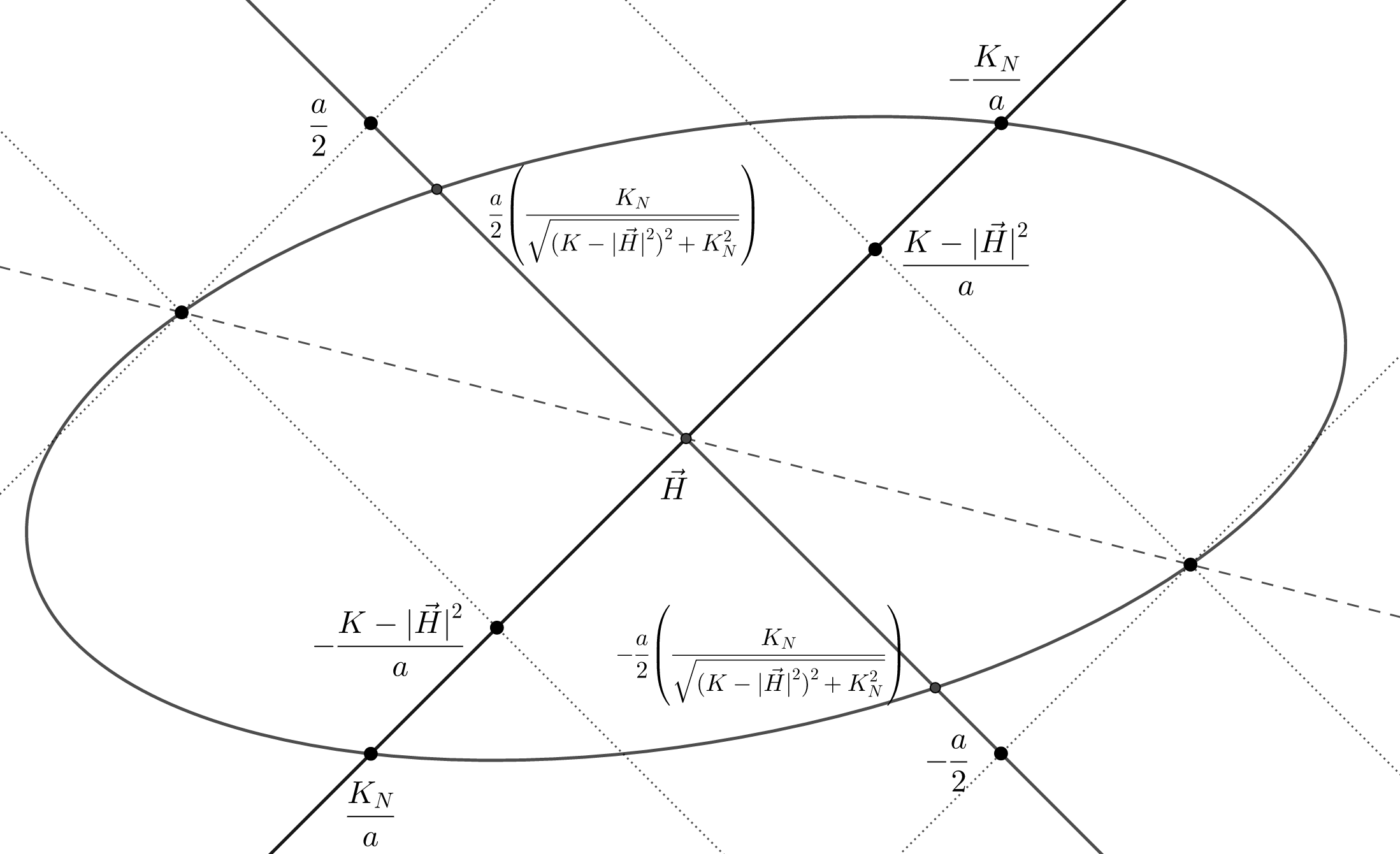}
\end{center}

\begin{rema} The normal vector $\vec{H}^*:=B(Z^\top,Z^\top)-\vec{H}$ satisfies $|\vec{H}^*|^2=2(|\vec{H}|^2-K).$ Moreover, since $a\neq 0,$ we have that $\vec{H}^*$ and $Z^\perp$ are linearly independent. Thus, if $|\vec{H}|^2\neq K,$ using the relation \eqref{ellipse}, we obtain the following characterization: in $(\vec{H},\vec{H}^*,Z^\perp),$ the equation of the ellipse is given by 
\begin{equation*}
\frac{x^2}{2| |\vec{H}|^2-K |}+\frac{y^2}{\frac{K_N^2}{a^2}}=1.
\end{equation*}
\end{rema}

We can also describe the curvature ellipse in the degenerate case where $a=0$ as in the following proposition. For briefness, we omit the proof.

\begin{prop} If $a=0,$ the curvature ellipse degenerates the segment 
$$\left[ \vec{H}+hZ^\perp, \vec{H}-hZ^\perp \right]$$ where $h=\max\{ \pm \la \vec{H}^*, W' \ra, \pm\beta(W) \}.$
\end{prop}

\paragraph{Mean curvature directions and asymptotic directions.}
If $p$ is a point on a spacelike surface $M$ in $\R^{3,1},$ a mean curvature direction in the tangent plane $T_pM$ is defined as the inverse image by the second fundamental form of the points in the ellipse of curvature where the line defined by the mean curvature vector  intersects the ellipse.

For all $p\in M$ and $u\in T_pM,$ the condition that determines the mean curvature direction is
\begin{equation*}
[\vec{H},B(u,u)]=0,
\end{equation*} where the brackets stand for the mixed product in $N_pM$ (the determinant in a positively oriented Lorentzian basis). 

\begin{lema}\label{lmcd} If $a\neq 0,$ in the basis $(Z^{\perp},W')$ for the normal bundle, we have
\begin{equation*}
\vec{H}=\frac{|\vec{H}|^2}{a} Z^\perp + \frac{a}{2} W'
\hspace{0.2in}\mbox{and}\hspace{0.2in}
B(u,u)=\left[u_1^2 \beta(Z^\top)+2u_1u_2 \beta(W)+u_2^2b\right]Z^\perp+(au_2^2) W'
\end{equation*} for all $u=u_1Z^\top+u_2 W \in T_pM.$
\end{lema}
\begin{proof}
We have the following:
\begin{equation*}
\vec{H}=\frac{1}{2}\left[ B(Z^\top,Z^\top)+B(W,W) \right]=\frac{1}{2} \left[ \beta(Z^\top)Z^\perp+bZ^\perp+aW'\right]=\frac{\beta(Z^\top)+b}{2}Z^\perp+\frac{a}{2}W',
\end{equation*} but, using Proposition \ref{relation}, we have 
\begin{equation*}
\frac{\beta(Z^\top)+b}{2}=\frac{\frac{K}{2}+b}{2}=\frac{K+ab}{2a}=\frac{|\vec{H}|^2}{a}.
\end{equation*} The expression for $B(u,u)$ is obtained from a direct calculation.
\end{proof}

\begin{prop} If $a\neq 0,$ the mean curvature directions are given by
\begin{equation*}
KZ^{\top}+\left(-K_N\pm \sqrt{K^2+K_N^2}\right)W.
\end{equation*} If $a=0,$ every tangent vector $u\in T_pM$ defines a mean curvature direction.
\end{prop}
\begin{proof} We suppose that $u=u_1Z^\top+u_2 W$ is a mean curvature direction, using the expressions of the Lemma \ref{lmcd} and Proposition \ref{relation}, we obtain
\begin{align*}
0=[\vec{H},B(u,u)]&= u_2^2 \left[ \frac{2|\vec{H}|^2-ab}{2} \right]-u_1^2 \frac{K}{2}+u_1 u_2 K_N \\
&= u_2^2 \frac{K}{2}-u_1^2 \frac{K}{2}+u_1 u_2 K_N ;
\end{align*} solving these equation we get
\begin{equation*}
u_2=\left(-K_N\pm \sqrt{K^2+K_N^2}\right) \frac{u_1}{K},
\end{equation*} taking $u_1=K$ we obtain the result.
\end{proof}

An asymptotic direction at $T_pM$ is defined as the inverse image of the second fundamental form of a point where the line that contains the origin is tangent to the ellipse of curvature. 

For all $p\in M,$ we consider the real quadratic form 
\begin{equation*}
\delta:T_pM \longrightarrow \mathbb{R},\hspace{0.2in} u\longmapsto dG_p(u)\wedge dG_p(u),
\end{equation*} where $\Lambda^4\mathbb{R}^{3,1}$ is identified with $\mathbb{R}$ by means of the volume element $e_1\wedge e_2\wedge e_3\wedge e_4\simeq 1.$ A non-zero vector $u\in T_pM$ defines an asymptotic direction at $p$ if $\delta(u)=0.$ The opposite of the determinant of $\delta,$ with respect to the metric in $M,$
\begin{equation*}
\Delta:=-\det \delta,
\end{equation*} is a second order invariant of the surface. There exist asymptotic directions if and only if $\Delta \geq 0$; moreover, $\Delta>0$  if and only if the surface admits two distinct asymptotic directions at every point. We refer to \cite[Section 4]{BS} for a complete description of the asymptotic directions of a spacelike surface in $\mathbb{R}^{3,1}.$ 

In the following proposition, we will compute the invariant $\Delta$ and describe the asymptotic directions of a spacelike surface with a canonical normal null direction.

\begin{prop}\label{asym} At every point of $M$ we have $\Delta=K^2,$
where $K$ is the Gauss curvature of $M.$ In particular, there are asymptotic directions at every point of $M.$
\end{prop}
\begin{proof} Since $\delta(u)=\Im m\ H(dG(u),dG(u))$ ($\delta$ is the imaginary part of the quadratic form $G^*H$) for all $u\in T_pM,$ using the equalities of Lemma \ref{disc} we have $\delta(Z^\top)=0$ and $$dG(Z^\top)\wedge dG(W)=\Im m\ H(dG(Z^\top),dG(W))=K.$$
By a direct computation we get $\Delta=[dG(Z^\top)\wedge dG(W)]^2-\delta(Z^\top)\delta(W)=K^2.$
\end{proof}

\begin{prop}\label{asym2} At every point of $M,$ $Z^\top$ is an asymptotic direction. Moreover, $W$ is an asymptotic direction if and only if $M$ has a flat normal bundle.
\end{prop}
\begin{proof} Since $\delta(u)=\Im m\ H(dG(u),dG(u))$ for all $u\in T_pM,$ by Lemma \ref{disc} we have 
\begin{equation*}
\delta(Z^\top)=0 \hspace{0.2in}\mbox{and}\hspace{0.2in} \delta(W)=-2K_N
\end{equation*} which implies the result.
\end{proof}

According to Proposition \ref{asym}, if the Gauss curvature $K$ is not zero, there exists two distinct asymptotic directions at every point of the surface. From Proposition \ref{asym2}, $Z^\top$ is an asymptotic direction; by a direct computation, in the following proposition we describe the other asymptotic direction. 

\begin{prop} If $a\neq 0,$ we have two cases to consider: when $K\neq 0,$ there exists two different asymptotic directions given by
\begin{equation*}
Z^\top \hspace{0.3in}\mbox{and}\hspace{0.3in} \dfrac{K_N}{K}Z^\top + W;
\end{equation*}
when $K=0,$ there exists a double asymptotic direction given by $Z^\top.$ 

If $a=0,$ every tangent vector $u\in T_pM$ defines an asymptotic direction.
\end{prop}
\begin{proof} We suppose that the other asymptotic direction is given by $rZ^\top+sW$ for some $r, s\in \R,$ with $s\neq 0.$ We have
\begin{align*}
0=\delta(rZ^\top+sW)&=dG(rZ^\top+sW)\wedge dG(rZ^\top+sW) \\
&= r^2 \delta(Z^\top)+s^2 \delta(W)+2(rs) dG(Z^\top)\wedge dG(W) \\
&= s^2 (-2K_N)+2(rs) K \\
&= 2s(-sK_N+rK)
\end{align*} (see the proof of Propositions \ref{asym}-\ref{asym2}), thus $r=\frac{K_N}{K}s;$ taking $s=1$ we obtain the result.
\end{proof}

\section{Construction of spacelike surfaces with a canonical normal null direction}\label{construction}

In this section, we construct some spacelike surfaces with a cononical normal null direction in the four dimensional Minskowski space $\mathbb{R}^{3,1}.$

\subsection{Translation Spacelike surfaces in $\mathbb{R}^{3,1}$} 
Let us consider the (translation) surface $M$ in $\mathbb{R}^{3,1}$ parametrized by
\begin{equation*}
\psi(x,y)=\alpha(x)+\delta(y)
\end{equation*}
where $\alpha$ and $\delta$ are two regular spacelike orthogonal curves in $\mathbb{R}^{3,1}$ such that $\alpha'$  and $\delta$ lies in the hyperplane orthogonal to canonical vector field $e_1=(1,0,0,0).$ 

In this case, the components of the induced metric $\langle\cdot,\cdot\rangle$ in $M$ are given by
\begin{equation*}
E:=\langle \psi_x,\psi_x \rangle=\langle \alpha'(x),\alpha'(x)\rangle=|\alpha'(x)|^2, \hspace{0.2in} F=\langle\psi_x,\psi_y\rangle=\langle\alpha'(x),\delta'(y)\rangle=0
\end{equation*} and 
\begin{equation*}
G=\langle \psi_y,\psi_y \rangle=\langle\delta'(y),\delta'(y)\rangle=|\delta'(y)|^2,
\end{equation*} thus, the determinant of this metric is
\begin{equation*}
\det\langle\cdot,\cdot\rangle=EG-F^2=|\alpha'(x)|^2|\delta'(y)|^2,
\end{equation*} in particular, since $\det\langle\cdot,\cdot\rangle>0,$ $M$ is spacelike (translation) surface in $\R^{3,1}.$

Now we study the conditions on $\alpha$ and $\delta$ such that the canonical vector field $e_1$ on $\mathbb{R}^{3,1}$ induce a canonical normal null direction on the surface $M.$ Writing $e_1=e_1^\top+e_1^\perp,$ we have 
\begin{equation*}
e_1^\top = c\psi_x+d\psi_y=c\alpha'(x)+d\delta'(y), 
\end{equation*} for some functions $c$ and  $d$ on $M.$ So, we have
\begin{equation*}
\langle e_1^\top,\alpha'(x) \rangle=c|\alpha'(x)|^2 \hspace{0.2in}\mbox{and}\hspace{0.2in}
0=\langle e_1,\delta'(y)\rangle=\langle e_1^\top,\delta'(y) \rangle=d|\delta'(y)|^2,
\end{equation*} therefore $d=0,$ and thus
\begin{equation*}
e_1^\top =\frac{\langle e_1^\top,\alpha'(x) \rangle}{|\alpha'(x)|^2}\alpha'(x).
\end{equation*}
\begin{prop} The canonical vector field $e_1$ induces a canonical normal null direction on the spacelike (translation) surface $M$ if and only if 
\begin{equation*}
\langle e_1^\top,\alpha'(x) \rangle^2=|\alpha'(x)|^2.
\end{equation*}
\end{prop}
\begin{proof} This result is a consequence of $|e_1^\top|^2=1.$
\end{proof}

Suppose that $M$ has a canonical normal null direction with respect $e_1.$ We can prove that $e_1^\top$ only depends on $x.$ Writing $\lambda(x)=\frac{\langle e_1^\top,\alpha'(x) \rangle}{|\alpha'(x)|^2},$ we have
$e_1^\top=\lambda(x)\alpha'(x),$ by Lemma \ref{levicivita}, we follow $\nabla_{e_1^\top} e_1^\top=0,$ thus
\begin{equation}\label{alpha1}
\nabla_{\alpha'(x)} \alpha'(x)=-\frac{\lambda'(x)}{\lambda(x)}\alpha'(x).
\end{equation} Analogously, since $W:=\frac{\delta'(y)}{|\delta'(y)|}$ is a unit spacelike vector field tangent to $M$ and orthogonal to $e_1^\top,$ from Lemma \ref{levicivita} we have $\nabla_{e_1^{\top}} W=0,$ thus 
$\lambda(x) \nabla_{\alpha'(x)} \frac{\delta'(y)}{|\delta'(y)|}=0$
{\it i.e.,} 
\begin{equation*}
\left( \frac{1}{|\delta'(y)|} \right)_x \delta'(y)+\frac{1}{|\delta'(y)|} \nabla_{\alpha'(x)} \delta'(y)=0,
\end{equation*} therefore
\begin{equation}\label{alpha2}
\nabla_{\alpha'(x)} \delta'(y)=0.
\end{equation}

On the other hand, since $\nabla_{W} e_1^\top=aW,$ we obtain
\begin{equation*}
\nabla_{\frac{\delta'(y)}{|\delta'(y)|}} \lambda(x)\alpha'(x)=a\frac{\delta'(y)}{|\delta'(y)|}
\end{equation*} or 
\begin{equation*}
0=\lambda_y(x)\alpha'(x)+\lambda(x)\nabla_{\delta'(y)} \alpha'(x)=a\delta'(y),
\end{equation*} thus we conclude that $a=0.$ In particular, from Corollaries \ref{normal} - \ref{gauss} we have that $M$ is flat and has a flat normal bundle.

Furthermore, by the mean curvature vector $\vec{H},$ given in \eqref{mcv}, the surface $M$ is minimal ({\it i.e.} $\vec{H}=0$ ) if and only if $B(e_1^\top,e_1^\top)=B(W,W)=0.$ Indeed, this vectors depends only on $x$ and $y,$ respectively. Using \eqref{alpha1}, we derive
\begin{equation*}
B\left(e_1^\top,e_1^\top\right)=\lambda^2(x)\alpha''(x)+\lambda(x)\lambda'(x)\alpha'(x)\neq 0
\end{equation*} if and only if $\lambda\alpha'$ is not a constant function.

For instance, the functions $\alpha(x)=(\cos x,\sin x,0,0)$ and $\delta(y)=(0,0,\sinh y,\cosh y)$ are spacelike orthogonal curves in Minkowski space $\R^{3,1},$ such that $\delta$ lies in the hyperplane orthogonal to $e_1.$ By the previous arguments, the surface $M$ parametrized by
\begin{equation*}
\psi(x,y)=(\cos x,\sin x,\sinh y,\cosh y),
\end{equation*} is a spacelike surface in $\R^{3,1}$ with a canonical normal null direction induced by $e_1;$ in this case we have
\begin{equation*}
e_1^\top=(\sin^2 x, -\sin x\cos x,0,0),
\end{equation*} which is not constant, therefore, $M$ is flat and has flat normal bundle, but it's not a minimal surface.

\subsection{Spacelike surfaces in $\mathbb{R}^{3,1}$ as a graph of a function}
We will study the situation where a spacelike surface is given as the graph of a smooth function.

Let $f,g:U\subset\mathbb{R}^2\to \mathbb{R}$ be two smooth functions and consider the surface
\begin{equation}\label{sup}
M:=\left\lbrace (x,y,f(x,y),g(x,y)) \in \mathbb{R}^{3,1} \mid (x,y)\in U \right\rbrace
\end{equation} 
given as the graph of the function $(x,y)\mapsto (f(x,y),g(x,y)).$ A global parametrization of the surface $M$ is given by $\psi:U\subset\mathbb{R}^2 \to \mathbb{R}^{3,1}$ which satisfies
\begin{equation*}
\psi(x,y)=(x,y,f(x,y),g(x,y)).
\end{equation*}
The tangent vectors to the surface $M$ are $\psi_x=(1,0,f_x,g_x)$ and $\psi_y=(0,1,f_y,g_y)$ and the components of the induced metric $\langle\cdot,\cdot\rangle$ in $M$ are given by
\begin{equation*}
E:=\langle\psi_x,\psi_x\rangle=1+f_x^2-g_x^2, \hspace{0.4in} 
F:=\langle\psi_x,\psi_y\rangle=f_xf_y-g_xg_y
\end{equation*} and
\begin{equation*}
G:=\langle\psi_y,\psi_y\rangle=1+f_y^2-g_y^2.
\end{equation*}
The determinant of this metric is
\begin{equation*}
\det\langle\cdot,\cdot\rangle=EG-F^2=\left( 1+|\nabla f|^2 \right)\left( 1-|\nabla g|^2 \right)+\langle \nabla f,\nabla g\rangle^2,
\end{equation*} where the right hand side is calculated on $\mathbb{R}^2$ with its standard Riemannian flat metric; in particular, $M$ is a spacelike surface if and only if $\det\langle\cdot,\cdot\rangle >0.$ 

\begin{prop}\label{e1} Let $M$ be a spacelike surface in $\mathbb{R}^{3,1}$ given as in \eqref{sup}. Then $M$ has a canonical normal null direction with respect to $e_1$ if and only if $EG-F^2=G.$ In this case we have
\begin{equation*}
e_1^\top=\psi_x-\frac{F}{G}\psi_y
\end{equation*}
\end{prop}
\begin{proof}
We need to compute the tangent part of $e_1$ along $M.$ Writing $$e_1^\top=a\psi_x+b\psi_y\ \in\ T_pM,$$ we get 
\begin{equation*}
1=\langle e_1,\psi_x\rangle=\langle e_1^\top,\psi_x\rangle=aE+bF \hspace{0.3in}\mbox{and}\hspace{0.3in} 0=\langle e_1,\psi_y\rangle=\langle e_1^\top,\psi_y\rangle=aF+bG.
\end{equation*} Thus, 
\begin{equation*}
a=\frac{G}{EG-F^2} \hspace{0.3in}\mbox{and}\hspace{0.3in} b=\frac{-F}{EG-F^2},
\end{equation*} therefore 
\begin{equation*}
 e_1^\top=\frac{G}{EG-F^2} \psi_x+\frac{-F}{EG-F^2}\psi_y.
 \end{equation*} Note that $e_1^\perp$ is a lightlike normal vector field if and only if $e_1^\top$ is a unit spacelike vector field along $M,$ thus, from the last equality we obtain
 \begin{align*}
  \langle e_1^\top,e_1^\top\rangle &=\frac{G^2}{(EG-F^2)^2}E-2\frac{FG}{(EG-F^2)^2}F + \frac{F^2}{(EG-F^2)^2}G \\
  &= \frac{G}{EG-F^2},
   \end{align*} that is, $e_1^\perp$ is a lightlike normal vector field along $M$ if and only if
$EG-F^2=G.$
\end{proof}

\begin{rema} In a similar way, if $M$ is a spacelike surface in $\mathbb{R}^{3,1}$ given as in \eqref{sup}, then $M$ has a canonical normal null direction with respect to $e_2$ if and only if $EG-F^2=E.$ In this case we have
$e_2^\top=-\frac{F}{G}\psi_x+\psi_y.$
\end{rema}

\begin{coro} With the same hypothesis as in Proposition \ref{e1},  $M$ has a canonical normal null direction with respect to $e_1$ if and only if
\begin{equation}\label{edp}
(1+f_y^2)g_x^2-2f_xf_yg_xg_y -(1-g_y^2)f_x^2 =0.
\end{equation}
\end{coro}
\begin{proof}
The condition $EG-F^2=G$ is equivalent to the equation \eqref{edp}.
\end{proof}

\subsubsection{Particular solutions of the PDE \eqref{edp}}

In order to find some particular solutions of the PDE \eqref{edp} we use conformal functions over the Lorentz numbers $\A=\{ x+\sigma y\mid x,y\in\R, \sigma\notin\R, \sigma^2=1 \},$ see for example \cite{VP1,K}.

We consider the función $h:\R^2\to\A,$ over the Lorentz numbers $\A$ given by $$h(x,y)=f(x,y)+\sigma g(x,y),$$ where $f$ and $g$ are the same functions of previous section; writing $z=x+\sigma y,$ $\overline{z}=x-\sigma y,$ $|z|^2=z\overline{z}$ and the operators $$\frac{\partial}{\partial z}=\frac{1}{2}\left( \frac{\partial}{\partial x}+\sigma \frac{\partial}{\partial y} \right) \hspace{0.3in}\mbox{and}\hspace{0.3in} \frac{\partial}{\partial \overline{z}}=\frac{1}{2}\left( \frac{\partial}{\partial x}-\sigma \frac{\partial}{\partial y} \right);$$ we easily get that the equation \eqref{edp} is equivalent to
\begin{equation*}
\left( \Big| \frac{\partial h}{\partial z} \Big|^2 - \Big| \frac{\partial h}{\partial \overline{z}} \Big|^2 \right)^2=
\Big| \frac{\partial h}{\partial z} + \frac{\partial h}{\partial \overline{z}} \Big|^2.
\end{equation*}
In particular, if $h:\R^2\to\A$ is a conformal function, {\it i.e.,} $\frac{\partial h}{\partial \overline{z}}=0,$ we get
\begin{equation*}
\Big| \frac{\partial h}{\partial z} \Big|^2=0 \hspace{0.4in}\mbox{or}\hspace{0.4in} \Big| \frac{\partial h}{\partial z} \Big|^2=1;
\end{equation*} in the first case, the components $f$ and $g$ of the function $h$ are given by
\begin{equation*}
f(x,y)=\frac{\alpha(x+y)+k}{2} \hspace{0.4in}\mbox{and}\hspace{0.4in} g(x,y)=\frac{\alpha(x+y)-k}{2},
\end{equation*} where $\alpha$ is some real function and $k$ is a constant. In the second case, the components $f$ and $g$ of the function $h$ are given by
\begin{equation*}
f(x,y)=\frac{\alpha^2(x+y)+x^2-y^2+\int c(x+y)d(x+y)+k(x-y)}{2\alpha(x+y)}
\end{equation*} and
\begin{equation*}
g(x,y)=\frac{\alpha^2(x+y)-x^2+y^2-\int c(x+y)d(x+y)-k(x-y)}{2\alpha(x+y)}, 
\end{equation*} for some real función $\alpha, c$ and $k.$

%\paragraph{Acknowledgements.}

\end{document}